\documentclass[article,11pt]{amsart}

\usepackage[usenames]{color}
\usepackage{graphicx}   
\usepackage{graphics,amssymb}  
\usepackage{amsfonts}
\usepackage{latexsym}
\usepackage{mathrsfs}
\usepackage{bbm}
\usepackage{hyperref}
\usepackage{tikz}

\newtheorem{theorem}{Theorem}[section]
\newtheorem{lemma}[theorem]{Lemma}

\newtheorem{proposition}[theorem]{Proposition}
\newtheorem{corollary}[theorem]{Corollary}

\newtheorem{definition}{Definition}[section]

\newcommand{\beq}{\begin{equation*}}
\newcommand{\eeq}{\end{equation*}}
\newcommand{\beqlbl}{\begin{equation}}
\newcommand{\eeqlbl}{\end{equation}}
\newcommand{\ba}{\begin{align*}}
\newcommand{\ea}{\end{align*}}

\newcommand{\field}[1]{\mathbb{#1}}

\newcommand{\matbegin}[1]{\left (  \begin{array} {#1} }
\newcommand{\matend}{ \end{array} \right ) } 
\newcommand{\prob}{\field{P}}
\newcommand{\expect}{\field{E}}
\newcommand{\var}{\text{Var}}

\newcommand{\bin}{\text{Bin}}
\newcommand{\dis}{\text{dis}}
\newcommand{\spn}[1]{\langle {#1} \rangle}
\newcommand{\inspn}{\mathcal{I}}
\newcommand{\indc}{\mathbf{1}}
\newcommand{\subtorusexists}{\mathcal{C}}
\newcommand{\sublat}{\mathcal{C}}

\newcommand{\perfect}[1]{{\mathcal{L}^*_{#1}}}
\newcommand{\torus}{\mathcal{F}}
\newcommand{\bigo}[1]{O\left( {#1}\right)}
\newcommand{\ceil}[1]{\lceil #1 \rceil}
\newcommand{\floor}[1]{\lfloor #1 \rfloor}

\newcommand{\nbr}{\mathcal{N}}
\newcommand{\binspn}{\mathcal{J}}
\newcommand{\vleft}{\mathcal{V}}
\newcommand{\vright}{\mathcal{W}}

\usepackage{geometry} 
\begin{document}
\begin{abstract}
	{This paper analyzes various questions pertaining to bootstrap percolation on the $d$-dimensional Hamming torus where each node is open with probability $p$ and the percolation threshold is 2.  For each $d'<d$ we find the critical exponent for the event that a $d'$-dimensional subtorus becomes open and compute the limiting value of its probability under the critical scaling.  For even $d'$, we use the Chen-Stein method to show that the number of $d'$-dimensional subtori that become open can be approximated by a Poisson random variable.}
\end{abstract}\title{Bootstrap Percolation on the Hamming Torus with Threshold 2}

\author{Erik Slivken} 


\maketitle


Bootstrap percolation first appeared in a paper by Chalupa et al \cite{bethe} as a model for ferromagnetism.  Adler et al\cite{brazil} provide a wonderful introduction to the subject.  

The process takes place on a graph $G=(V,E)$ with vertex set $V$ and edge set $E$ and depends on a parameter $\theta$ which we call the {\it threshold}.  Each vertex in the graph is initialized to one of two states, either open or closed.  At each subsequent step a vertex becomes open if at least $\theta$ of its neighbors are open.  Once open, a vertex remains open.  

Let $\omega \in \{0,1\}^V$ denote a configuration of the vertices.  If a vertex $v\in V$ satisfies $\omega(v)=1$, we say $v$ is open.  Similarly, if $\omega(v) = 0$, we say the vertex is closed.  For bootstrap percolation with threshold $\theta$ and initial configuration $\omega_0$, we construct a sequence of configurations $\{\omega_t\}_{t\geq 0}$ as follows:

\beqlbl
\omega_{t+1}(v) = \left \{\begin{array}{rl}
1, & \omega_t(v) = 1\text{ or } \sum_{v' \sim v} \omega_t(v') \geq \theta\\
0, & \text{otherwise}
\end{array} \right . 
\eeqlbl 
where $v' \sim v$ if there is an edge in $E$ connecting $v$ and $v'.$

In this paper we will assume that the probability that $\{\omega_0(v)\}_{v\in V}$ are independent Bernoulli($p$) random variables for each $v$.  Given some initial configuration, we can ask what the evolved configuration will look like after some time.  In particular we care about the steady state, $\omega_\infty := \limsup_{t\to \infty}\omega_t.$  Given a distribution on $\omega_0$ what can we say about $\omega_\infty?$

The first rigorous results came from van Enter \cite{vanenter} and later Schonmann \cite{schonmann}.  They showed that there is no non-trivial phase transition on the infinite lattice $\mathbb{Z}^d$ with edges connecting each vertex to its $2d$ nearest neighbors. For $\theta\leq d,$ if $p>0$, then with probability $1$, every point eventually becomes open.  If $\theta > d$ then everything becomes completely open with positive probability only if $p=1.$

The next big step in the history of bootstrap percolation was to view the process on a family of finite graphs $\mathcal{G} = \{G_n = (V_n,E_n)\}$ where the probability that a vertex is initially open is given by a function of $n$, $p(n)$.  As each graph is finite, for any increasing event $A$, $f_A(p(n)):=\prob_{p(n)}(A)$ is an increasing polynomial in $p(n)$ with $f_A(0) = 0$ and $f_A(1) =1$.  By continuity, for each $\alpha \in [0,1]$ there is some $p_\alpha(A,n)$, such that $f_A(p_\alpha(A,n))=\alpha.$  As is customary, we let $p_c(A,n)$ denote the critical probability $p_{1/2}(A,n).$  For the remainder of the paper we will suppress the dependence on $n$ so that $p = p(n)$ and, similarly, $p_c(A) = p_c(A,n)$ or $p_c = p_c(n)$ when $A$ is unambiguous.  All limits will be as $n$ tends to infinity unless otherwise specified.        

We say there is a sharp phase transition for an increasing event $A$ if a small perturbation from the critical probability drastically changes the probability of $A$.  More formally the phase transition is sharp if for any $\epsilon \in (0,1),$
$$p_{1-\epsilon} -p_{\epsilon} = o(p_c).$$ Friedgut and Kalai \cite{FriKal} investigate this phenomenon in some generality.  

For an increasing event $A$, we say that $\gamma$ is a critical exponent for $A$ if for any $\epsilon > 0$ $$\prob_p( A ) \to \left\{ \begin{array}{lr} 1, & p >n^{-\gamma+\epsilon}\\0, & p<n^{-\gamma-\epsilon}\end{array}\right . .$$  If $\gamma$ is the critical exponent of $A$, then for all $\epsilon >0$, $n^{-\gamma-\epsilon} < p_c(A) < n^{-\gamma + \epsilon}.$

Many results concern the event $\sublat:=\{\omega_\infty \equiv \indc\}$ and the corresponding critical probability $p_c = p_c(\sublat).$  Aizenman and Lebowitz \cite{aizenman} showed for the finite $d$-dimensional grid, $[n]^d$, and threshold $\theta=2$, there exists constants $c_1,c_2$ such that $c_1< (\log n)^{d-1} p_c < c_2.$  Moreoever, they show that the phase transition is sharp.  

In a widely celebrated paper Holroyd \cite{holroyd} showed that for $d=\theta =2$ $$p_c\sim \pi^2/18\log n.$$   Later this result was expanded by Holroyd, Ligget, and Romik \cite{hlr} \ to $d=2, \theta =k+1$ where the neighborhood of a vertex is the $k$ closest vertices in each of the cardinal directions.  They show $p_c\sim \pi^2/(3(k+2)(k+1)\log n)$ for this graph.  These types of results have been extended to higher dimensions by \cite{bbmhighdim}, hypercubes \cite{BB:2006}, random graphs \cite{bolaghpittel}, and more geometric settings \cite{bollobasriordan}.  This is a very active area of research.
   
Our graph of interest is the $d$-dimensional Hamming torus.  The Hamming torus has the same vertex set as the finite $d$-dimensional grid, $V = [n]^d$, but the edge set is modified so that $$E := \{ ( v,w ) :  v  \text{ differs from } w \text{ in exactly one coordinate } \}.$$
Gravner et al. \cite{ghps} introduced the study of bootstrap percolation on the Hamming torus.  For general thresholds $\theta\geq 2$ they investigate the critical probability, $p_c$.  The large neighborhood size of a vertex in the Hamming torus makes the behavior of $p_c$ rather different from that of the nearest neighbor counterparts.  Their results suggest $p_c$ is on the order of $n^{-\alpha}$ for some positive constant $\alpha.$    

  They also consider finer structure, which we now introduce.  
\begin{definition}\label{def.subtorus}
A subset $V\subset [n]^d$ is a {\bf subtorus} if there exists a set of indices $I(V)$ and constants $\{ \alpha_l\}_{l\in I(V)}$ such that $v\in V$ if and only if for all $l \in I(V), v_l=\alpha_l.$  For fixed $d$, we say $V$ has dimension $i$ if $|I(V)| = d-i$ and denote by $\torus_i$ the collection of all such subtori.
  
\end{definition} 
For $0 \leq i \leq  d$, they study the events
$$\sublat_{i} = \{ \exists   V \in \torus_i\ s.t. \ \omega_\infty|_V \equiv \indc \}.$$ Following the notation in \cite{ghps} let $p_c(\theta,i,d)$ be the critical probability for the event $\sublat_i$ on the $d$-dimensional Hamming torus.   Gravner et al. show for $d=2$ and any $\theta\geq 2$ that $p_c(\theta,1,2) = p_c(\theta,2,2)$ and for any $p$, $$\prob_{p}(\{ \omega_\infty \not\equiv \omega_0 \} \backslash \{\sublat_d\}) = o(1).$$ 

For $d=\theta=3,$ and $p=an^{-2}$ they compute a precise limiting value of $\prob_p(\sublat_3)$ that varies continuously from $0$ to $1$ as $a$ increases from $0$ to infinity.  In particular, the transition is not sharp.  For larger $d$ and $\theta$ they prove upper and lower bounds on the critical exponent for $\sublat_d$, provided it exists.  For large enough $d$ and $\theta$ they show this is different than the critical exponent of $\sublat_1.$  

We consider the case $\theta =2$ and $d>2.$  The case where $d=2$ is well understood. (See Figure \ref{fig1} for a picture of the process with $d=\theta=2$).  We give a very precise description of the fine structure of this dynamics.

For fixed $d>2$, define $$J_d = \max\{j : j(j+1) < d\}.$$
We show that the critical exponents for $\sublat_2,$ $\sublat_4, \cdots, \sublat_{2J_d}$ are distinct.  We also show for every $j$ such that $2\leq 2j \leq d, $ the critical exponent for $\sublat_{2j}$ and $\sublat_{2j-1}$ are the same and for any $p\in [0,1]$, $$\prob_p(\sublat_{2j-1} \backslash \sublat_{2j} ) \to 0.$$  

If $(J_d+1)(J_d+2) >d$, then, for all $p\in [0,1]$, we have $\prob_p(\sublat_{2J_d} \backslash \sublat_{d}) \to 0.$  Whereas if $(J_d+1)(J_d+2)=d$ then $\sublat_{2J_d}$ and $\sublat_{d}$ have the same critical exponent, but for certain values of $p$, $\prob_p(\sublat_{2J_d}\backslash\sublat_d)$ is bounded above a small positive constant when $\prob_p(\sublat_{2J_d})$ has a positive limit.

  After we determine the critical exponent for these events, we give a precise description of the asymptotics of $p_c(\sublat_i).$  Unlike the threshold functions for the grid $[n]^d$ found in 
  \cite{bbmhighdim}, $p_c(\sublat_i)$ is not sharp.  Understanding these precise asymptotics helps with understanding how a typical configuration evolves which in turn should be useful when studying larger $\theta.$     
  
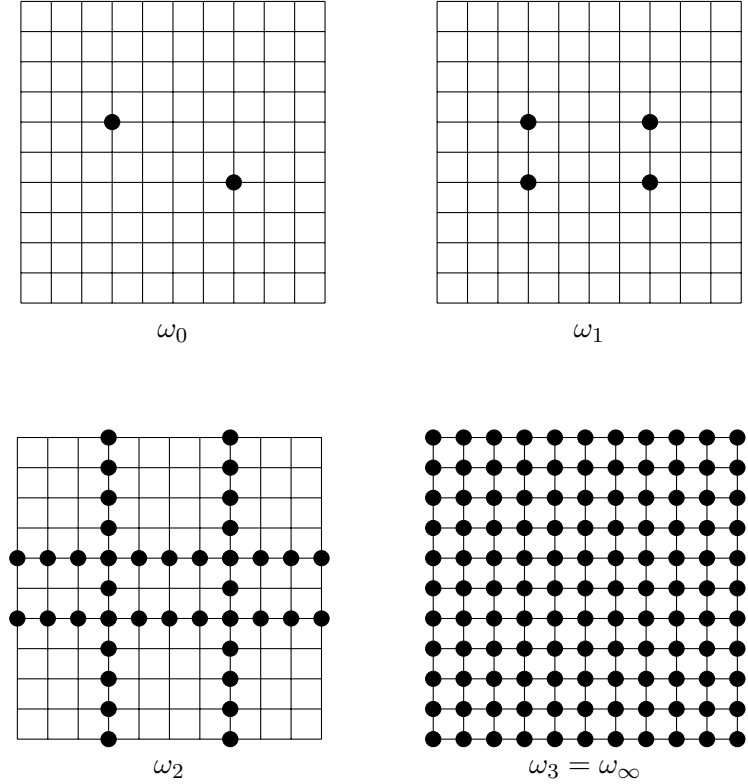
\begin{figure}\centering

\begin{tikzpicture}[scale = .4]
\filldraw[white] (0,0) circle (.25cm);
\filldraw[white] (10,0) circle (.25cm);
\draw (0,0) grid (10,10);
\filldraw[black] (3,6) circle (.25cm);
\filldraw[black] (7,4) circle (.25cm);
\draw (5,-1) node {$\omega_0$};
\end{tikzpicture}
\hspace{1cm}
\begin{tikzpicture}[scale = .4]
\filldraw[white] (0,0) circle (.25cm);
\draw (0,0) grid (10,10);
\filldraw[black] (3,6) circle (.25cm);
\filldraw[black] (7,4) circle (.25cm);
\filldraw[black] (7,6) circle (.25cm);
\filldraw[black] (3,4) circle (.25cm);
\draw (5,-1) node {$\omega_1$};
\end{tikzpicture} 

\vspace{1cm}
\begin{tikzpicture}[scale = .4]
\draw (0,0) grid (10,10);
\foreach \i in {0,...,10}
{
\filldraw[black] (3,\i) circle (.25cm);
\filldraw[black] (7,\i) circle (.25cm);
\filldraw[black] (\i,4) circle (.25cm);
\filldraw[black] (\i,6) circle (.25cm);
}
\draw (5,-1) node {$\omega_2$};
\end{tikzpicture}
\hspace{1cm}
\begin{tikzpicture}[scale = .4]
\draw (0,0) grid (10,10);
\foreach \i in {0,...,10}
{
	\foreach \j in {0,...,10}
	{
	\filldraw[black] (\j,\i) circle (.25cm);
	}
}
\draw (5,-1) node {$\omega_3 = \omega_\infty$};
\end{tikzpicture}

\caption{The bootstrap percolation process with threshold $\theta=2$ starting with two non-colinear open nodes.}
\label{fig1}
\end{figure}
\section{Statements}

First, we need a few definitions.  We will identify $\omega_t$ with the set $\{v: \omega_t(v) =1\}.$ 

\begin{definition}\label{def.internallyspanned}

For a set of nodes, $S$, we define their {\bf span}, $\spn{S},$ to be the set $\omega_\infty$ of eventually occupied points starting from $\omega_0 =S.$  We say $V$ is {\bf internally spanned} by $S$ if $V = \spn{S\cap V}.$  

\end{definition}

For arbitrary $\omega_0$ we consider the following events:

\begin{itemize}
\item $\inspn_{V} = \{\omega_0 \text { internally spans } V\}$,
\item $\inspn_{i}  = \{ \exists \  V \in \torus_i \ s.t. \ \inspn_V$ occurs $\} = \bigcup_{V\in\torus_i} \inspn_V,$ 
\item $\sublat_{i}  = \{ \exists \  V\in \torus_i \ s.t.\ \omega_\infty|_V \equiv \indc \}.$
\end{itemize}

Note the slight difference in the definitions of $\inspn_i$ and $\sublat_i$.  For $\sublat_i$ the only thing that matters is the final state $\omega_\infty$ where for $\inspn_i$ it is important how one gets to $\omega_\infty.$   

For the remainder of this paper we drop the parameter $\theta$ as it will always be $2$.  Throughout the paper we will assume $d>2$ as that case was answered completely for all $\theta$ in \cite{ghps}.  For $d>2$, and $0\leq i \leq d$ denote the threshold functions of $\inspn_i$ and $\sublat_i$ by $p_c(\inspn_i)$ and $p_c(\sublat_i)$ respectively.  Much of the work in this paper is in finding bounds for the threshold function for $\inspn_i.$  Then we show that $p_c(\sublat_i) $ will have the same asymptotic behavior as $p_c(\inspn_i)$ when $i$ is even.

Now we are in a position to state our main results.  To shorten the statements of the following theorems we define $$\lambda(j,d,a) := {d \choose 2j}(2j)!2^{-j-1}a^{j+1}.$$
\begin{theorem} \label{theorem1}

Fix $d>2$ and $j\leq J_d,$ and let $p = an^{-d/(j+1)-j}.$ Then 

\beqlbl  \label{part1}
\prob_p(\inspn_{2j})\to 1 - e^{-\lambda(j,d,a)},
\eeqlbl
and
\beqlbl \label{part2}
\prob_p( \sublat_{2j}\backslash \inspn_{2j} ) \to 0.
\eeqlbl

\end{theorem}

In fact to prove Theorem \ref{theorem1} part \ref{part1} we prove a stronger result on Poisson convergence by an application of the Chen-Stein method \cite{ross.steinmethod}.  For two non-negative integer valued random variables $Y$ and $Z$ the total variation is defined as $$d_{TV}(Y,Z) = \frac{1}{2}\sum_{k=0}^\infty |\prob( Y=k ) - \prob( Z=k )|.$$ 

\begin{theorem} \label{theorem3}
Fix $d > 2.$ Let $ j \leq J_d$, $p =an^{-d/(j+1) - j},$ and $\lambda(j,d,a).$  Let $Y_j$ denote the number of subtori $V\in \torus_{2j}$ such that  $\inspn_V$ occurs, and let $Z_j$ denote a Poisson($\lambda(j,d,a)$) random variable.  Then $$\lim_{n\to \infty} d_{TV}(Y_j,Z_j) \to 0.$$

\end{theorem}

The precision given by Theorem \ref{theorem3} leads to the following results:

\begin{theorem} \label{theorem4}

Fix $d>2$ such that $d<(J_d+1)(J_d+2)$ and let $p=an^{-d/(J_d+1) - J_d }$.   Then

$$ \prob_p(\inspn_{2J_d} \backslash \inspn_d )\to 0,$$ so $$\prob_p(\sublat_d) =\prob_p(\inspn_d ) \to 1- e^{-\lambda(J_d,d,a)}.$$

\end{theorem}

\begin{theorem} \label{theorem5}

Fix $J_d\geq 1$ and let  $d = (J_d+1)(J_d+2),$ $p = an^{-2J_d-2}.$  There exists positive constants $0<c_1,c_2<1 - e^{-\lambda(J_d,d,a)}$ such that for all large enough $n$

\begin{align}
&\prob_p(\inspn_{2J_d+2} ) > c_1\label{thm51} \\
& \prob_p( \inspn_{2J_d} \backslash \inspn_{2J_d+2} ) > c_2 \label{thm52}\\
\intertext{ and }
& \prob_p(\inspn_{2J_d+2} \backslash \inspn_{d} ) \to 0. \label{thm53}
\end{align}
\end{theorem}

The following theorem highlights how $J_d=1$ is different from higher $J_d$ when $d= (J_d+1)(J_d+2)$.
\begin{theorem} \label{theorem6}

Fix $J_d \geq 1$ and let $d = (J_d+1)(J_d+2)$ and $p = an^{-2J_d-2}.$ If $J_d>1$ then
\begin{align}
&\prob_p( \inspn_{d} \backslash \inspn_{2J_d+2} )\to 0, \label{thm62}\\
\intertext{whereas if $J_d=1$, then there exists $c>0$ such that for large enough $n$, }
&\prob_p( \inspn_{6} \backslash \inspn_{4} ) > c. \label{thm63}
\end{align}
\end{theorem}

In Section 2, we prove lemmas that describe the evolution of $\omega_t$ when $\theta = 2.$  In Section 3, we prove upper and lower bounds for the probabilities of the events $\sublat_{2j}$ and $\inspn_{2j}$.  In Section 4 we use the Chen-Stein method \cite{ross.steinmethod} to describe precisely the asymptotics of $p_c(\sublat_i,d)$ and $\prob(\inspn_{2J_d}).$  In Section 5 we combine everything to prove our statements.  

%
%
%

\section{Deterministic Results}

We begin with the simplest case.  Suppose $u\neq v$ are the only nodes which are initially open.  Denote the Hamming distance between the nodes as $\dis(u,v) := \sum_{i=1}^d \indc_{u_i\neq v_i},$ the number of coordinates where $u$ and $v$ differ.  If $\dis(u,v)>2$ then no new nodes become open $\spn{\{u,v\}} = \{u,v\}$.  If $\dis(u,v) \leq 2$ then $u$ and $v$ must agree for all but at most $2$ indices.  Without loss of generality, we may assume that $u_i = v_i$ for $i > 2$.  

Suppose first that $u_2 = v_2$ as well ( i.e. $\dis(u,v) = 1)$, the line $\big\{(t, u_2, \cdots ), t\in [n]\big\}$ has two nodes initially open, and after one step every node in that line becomes open.  Every node not on the line has at most one neighbor on the line, so growth stops. 

If $\dis(u,v) = 2$, then after one step the common neighbors of $u$ and $v$,  $u' = ( u_1, v_2, \cdots)$ and $v' = (v_1, u_2, \cdots)$, become open.  The nodes $u$ and $u'$ are two different open neighbors for every closed node in the  line $\big \{( u_1, s, \cdots ): s\in [n]\big \},$ so after two steps the entire line becomes open.  The same is true for the lines containing both $u$ and $v'$, both $v$ and $u'$, and both $v'$ and $v$.  Once those lines are open every other node in the plane $\big\{ (t,s,\cdots) : (t,s) \in [n]^2\big\}$ has a at least two ( in fact four ) open neighbors, so the entire plane becomes open.  (See Figure \ref{fig1})

Growth for higher dimension subtori is a bit more involved.  First we generalize the distance function to subsets $S_1,S_2$ as follows, $$\dis(S_1,S_2) = \inf_{u\in S_1,v \in S_2}\dis(u,v).$$  

%
%
%

We will state and prove a few necessary lemmas.  The key point is that growth continues only if there are two sets of open nodes within distance 2 of each other.  

\begin{lemma}\label{lemma.spnptclosed} 

For $S\subset [n]^d$, let $\overline{S}$ denote the smallest subtorus that contains $S$.  If $V$ is a subtorus and $u$ is a node with $\dis(V,u)\leq 2$ then  $$\spn{V\cup \{u\}} = \overline{V\cup\{ u\}}.$$

\end{lemma}   

\begin{proof}{(By induction on $i = \dim(V)$)}  We have shown that the lemma holds if $V$ has dimension $0$ (a single node).  Suppose the lemma holds for all subtori $W$ with $\dim(W) < i$.  Let $V$ be a subtorus with $\dim(V) = i$ and let $u$ be a node with $\dis(V,u)\leq2.$  Without loss of generality we assume the last $d-i$ coordinates are fixed, i.e. $I(V) = [i+1, d].$  Without loss of generality we may also assume that 

$$u \in \{ ( u_1, \cdots ,u_d ) : u_{l} = \alpha_l(V) \text{ for } l > i+2 \}.$$ Let $V_k$ denote the subtorus of $V$ that fixes the $k^{th}$ coordinate to the value $u_k$.  Then $V_k$ has dimension $i-1$ and $\dis(V_k,u)\leq 2$.  By the induction hypothesis, $\spn{V_k,u} = \overline{V_k\cup\{u\}}$.  For $a = (a_1,\cdots,a_d) \in \overline{V\cup \{u\}},$ there are two neighbors $$b= (u_1, a_2, \cdots, a_d)\in \overline{V_1\cup \{u\}} $$ and $$ c = (a_1, u_2, \cdots, a_d) \in \overline{V_2\cup\{u\}},$$ so $a$ becomes open and we can conclude $\overline{V\cup \{u\}}\subseteq \spn{V_1\cup V_2\cup \{u\}}\subseteq \spn{V\cup \{u\}}$.  By monotonicity $\spn{V\cup \{u\}} \subseteq \spn{\overline{V\cup \{u\}}} = \overline{V\cup \{u\}}$ so we have equality for the two sets.  Moreover, if $u\notin V$ then $i+1 \leq \dim( \overline{V\cup\{u\}} ) \leq i+2$.
\end{proof}

\begin{lemma}\label{lemma.spnclosed}

If $V, W$ are open subtori and $\dis(V,W) \leq 2$ then $\spn{ V\cup W} = \overline{V\cup W}$.  
\end{lemma}

\begin{proof}

This is a natural extension of Lemma \ref{lemma.spnptclosed}.  By monotonicity we have $\spn{V\cup W}\subseteq \spn{\overline {V\cup W } } = \overline{V\cup W}.$  Let $V^0 = V$.  We define $V^l$ recursively.  Let $W^{l-1}$ denote the subset of $W$ that satisfies $0<\dis(V^{l-1},u) \leq 2$ for every $u \in W^{l-1}.$    For $l>0$ if $W\cap (V^{l-1})^c$ is non-empty there exists a $w_l\in W^{l-1}$.  We then define $V^l = \spn{V^{l-1},w_l}$ for some choice of $w_l$.  By Lemma \ref{lemma.spnptclosed} this is the subtorus $\overline{V^{l-1}\cup\{w_l\}}.$ Its dimension is strictly greater than $\dim(V^{l-1})$.  If $W\cap (V^{l-1})^c$ is empty then $V^l = V^{l-1}$.   

Since $\{V^l\}$ is an increasing sequence of subtori bounded by $\overline{V\cup W}$ it must stabilize to some subtorus $V^m$ in a finite number of steps.  By definition $V\subseteq V^m$, and more importantly, $W \cap (V^m)^c = \emptyset$ so $W\subseteq V^m.$  Since $V^m = \spn{V\cup \{ w_1, \cdots, w_m\} }$ we also have that $V^m \subseteq \spn{V,W}.$  Combining everything we get $$\overline{V\cup W} \subseteq V^m \subseteq \spn{V\cup W} \subseteq \overline{V\cup W}$$ and the lemma holds. 
\end{proof} 

\begin{definition}\label{def.maximal}

A subtorus $V$ is {\bf maximal} in $\spn{S}$ if no other subtorus in $\spn{S}$ contains $V$.  
\end{definition}

The next two lemmas give conditions for when and how a subtorus is internally spanned.

\begin{lemma} \label{lemma.maxinspn} For an initial configuration of open nodes $S$, let $V$ be a maximal subtorus in $\spn{S}.$   Then $V$ is internally spanned with $V = \spn{S\cap V}.$  
\end{lemma}

\begin{proof}
Let $S_1 = S\cap V$ and $S_2  = S \backslash S_1$.  If $\spn{S_1} = V$ then we are done.  Suppose that $\spn{S_1} \neq V$.  Since $V$ eventually becomes open, there must be some node $u \in \spn{S_2}$ such that $\dis(\spn{S_1},u) \leq 2,$ otherwise evolution would stop and $V$ could not be contained in $\spn{S}.$    In particular, there is a node $u\in\spn{S_2}$ such that $u\notin V$ yet $\dis( V, u ) \leq 2$.   By Lemma \ref{lemma.spnclosed} the smallest subtorus that contains both $u$ and $V$ becomes open eventually.   However $V$ is maximal so no such $u$ can exists and $\spn{S_1} = V$.     
\end{proof} 

\begin{lemma} \label{lemma.inspnbreakdown}

Let $S$ be a set of open nodes in $[n]^d$ with $V \subset \spn{S}$ a maximal open subtorus.  There exist disjoint non-empty subsets $S_1, S_2 \subset S$ and subtori $V_1, V_2 \subset V$ with $\dim(V_1)\leq \dim(V_2) < \dim(V)$ such that $\spn{S_1} = V_1$, $\spn{S_2}  = V_2,$ and $\spn{S_1 \cup S_2} = V.$

\end{lemma}

\begin{proof}

$V$ is maximal so we may assume $\spn{S} = V$.   Consider the sequence of nested collections of subtori contained in $\spn{S}$, $$\{ W_i^0\} \subset \{ W_i^1 \} \subset \cdots \subset \{W_i^k \} \subset V$$ where $S = \{W_i^0\}$ and $\{W_i^{k+1} \}$ is formed by finding two subtori $W_{i_1}^k$ and $W_{i_2}^k$ within Hamming distance 2 of each other and setting $W_{i_1}^{k+1} = \spn{W_{i_1}^k \cup W_{i_2}^k}$ and reindexing the others appropriately.  Since $S$ is finite, eventually we will have two subtori $W^{k}_{i_1}, W^k_{i_2} \neq V$ such that $\spn{ W^k_{i_1} \cup W^k_{i_2} } = V$.  Each $W_{i_l}^k$ had a set $S_l$ such that $\spn{S_l}  = W_{i_l}^k$ for $l = 1,2$.  
\end{proof}

\section{Critical Probability}

To find the asymptotics of $p_c(2j,d)$, we will first prove upper and lower bounds for the exponent of $p_\inspn(2j,d).$  Since $\inspn_{2j} \subset \sublat_{2j}$ any upper bound for $p_\inspn(2j,d)$ will hold for $p_c(2j,d).$  With a little more work, we then prove the lower bound for the exponent of $p_\inspn(2j,d)$ will also be a lower bound for the exponent of $p_c(2j,d).$   

For odd dimension subtori we will show that $\prob_p(\inspn_{2j-1}) \leq (1+o(1))\prob_p(\inspn_{2j})$ hence asymptotically $p_c(2j-1,d) \sim p_c(2j,d).$  This is apparent in the case of a line and a plane.  For a line to be internally spanned, two open nodes need to be co-linear, whereas for a plane to be internally spanned, two open nodes only need to be co-planar.

 \subsection{ Upper Bounds for $p_c(i,d)$ and $p_\inspn(i,d)$ }
 
 For fixed $d$ and $p$, the probability of $\inspn_V$ is identical for $V\in \torus_{i}.$  We then denote for any particular $V\in \torus_i$
\beqlbl \label{midef}
M_{i}:=\prob_p(\inspn_{V} ).
\eeqlbl
 
 \begin{lemma} \label{lemma.ealpha} 

Fix $\epsilon >0$ and suppose $p < n^{-2J_d-\epsilon}$.  For $1\leq i\leq J_d$, there exists a constant $c_{d,\epsilon} > 0$ such that for every $V \in \torus_{2i}$ and $n$ large
\beq
\prob_p(\inspn_V) = M_{2i} \geq (2i)!2^{-i-1}n^{i(i+3)}p^{i+1}(1-n^{-c_{d,\epsilon}}).
\eeq
\end{lemma}

\begin{proof}[Proof of Lemma \ref{lemma.ealpha}]

Let $V$ be a subtorus with dimension $2i$.  Suppose we have a collection of distinct nodes $S = \{v_1, \cdots, v_{i+1} \} \subset V$ such that $\spn{\{v_1, \cdots , v_{i+1}\} } = V.$  The probability that only these nodes are open is exactly $p^{i+1}(1 - p)^{n^{2i}-i-1}.$  Let $\mathcal{L}_V$ be the set of all such collections.  Since $p < n^{-2J_d-\epsilon}$ and $i \leq J_d$ there exists constant $\beta_{d,\epsilon} > 0$ such that $(1-p)^{n^{2i}-i-1} \geq(1-n^{-\beta_{d,\epsilon}})$ for sufficiently large $n$.  Then 

\beqlbl \label{m2ilower}
M_{2i} \geq \sum_{\mathcal{L}_V} p^{i+1} (1-p)^{n^{2i} - i - 1}  \geq |\mathcal{L}_V|p^{i+1}(1-n^{-\beta_{d,\epsilon}}).
\eeqlbl

We call a ordered collection, $S = \{v_1, \cdots, v_{i+1}\}$ {\em perfect} in $V$ if the following are satisfied:

\begin{itemize}
\item $\spn{S} =V$,
\item for $1\leq i_1 < i_2 \leq i+1$, $\dis(v_{i_1}, v_{i_2}) = 2(i_2-1),$
\item and $v_1 < v_2$ in lexicographical ordering.
\end{itemize}   For $i' \leq i$, the subcollection $S_{i'} = \{ v_1, \cdots, v_{i' + 1} \}$ is also perfect in $\spn{S_{i'}}= V'$ and $\dim(V') = 2i'.$  Note that a non-trivial rearrangement of a perfect ordered collection is not a perfect ordered collection.  We call an unordered collection perfect if there exists an ordering of that collection that is perfect.

Let $\perfect{V}\subset \mathcal{L}_V$ denote the set of perfect collections for $V$.  We will show for $V$ in $\torus_{2i}$ there is a sequence of constants $\{b_i\}$ such that for large enough $n,$

\beqlbl \label{perfectsize}
|\perfect{V}| \geq (2i)!2^{-i-1}n^{i(i+3)}(1-b_in^{-1}).
\eeqlbl

Let $b_1 = 2$ and define recursively $b_i$ for $i\geq 2$ by the recursion $b_i = 4ib_{i-1}$.  For a plane, $P$, a pair of points is perfect if they are not collinear.  Hence
$$|\perfect{P}| = {n^2 \choose 2} - 2n {n \choose 2} \geq \frac{n^4}{2}(1-2n^{-1})$$
and Inequality \ref{perfectsize} is true.  We continue inductively and assume for $i\geq 2$ and a subtori $W\in \torus_{2i-2},$ $$|\perfect{W}| \geq (2i-2)!2^{-i}n^{(i-1)(i+2)}(1-b_{i-1}n^{-1}).$$  Suppose $W\subset V$ and a fix $S'\in \perfect{W}$, then $\{v\} \cup S'$ is in $\perfect{V}$ if $v\in V$ differs in the first $2i$ coordinates with each $w\in S'$ and agrees with the rest.  Therefore there are at least $(n-i)^{2i}$ possible choices of $v \in V$ where $\{v\} \cup S'$ is perfect.  For $V \in \torus_{2i},$ there are exactly ${2i \choose 2}n^2$ $W\subset V$ with $W \in \torus_{2i-2}$.  Then

\begin{align*}
|\perfect{V}| =& \sum_{W\subset V,  W\in \torus_{2i-2}} \sum_{S' \in \perfect{W}} \sum_{v \in V} \indc_{S' \cup \{v\}  \text{ is perfect in $V$ }}\\ \nonumber
\geq& \sum_{W\subset V,  W\in \torus_{2i-2}} \sum_{S' \in \perfect{W}} (n- i)^{2i}\\
\geq& \sum_{W\subset V,  W\in \torus_{2i-2}} ( 2i-2)!2^{-i}n^{(i-1)(i+2)}(1-b_{i-1}n^{-1})n^{2i}(1- in^{-1})^{2i}\\
\geq & {2i \choose 2}n^2(2i-2)!2^{-i}n^{(i-1)(i+2)}(1-4ib_{i-1}n^{-1})\\
\geq & (2i)!2^{-i-1}n^{i(i+3)}(1-b_in^{-1}).
\end{align*}

 Combining Inequalities \ref{m2ilower} and \ref{perfectsize} gives

\begin{align*}M_{2i}&\geq |\mathcal{L}_V|p^{i+1} (1-n^{-\beta_{d,\epsilon}})\\
&\geq |\perfect{V}|p^{i+1} (1-n^{-\beta_{d,\epsilon}})\\
&\geq (2i)!2^{-i-1}n^{i(i+3)}(1-b_in^{-1})p^{i+1}(1-n^{-\beta_{d,\epsilon}}).
\end{align*}
Therefore, for $n$ sufficiently large, $(1-n^{-\beta_d})(1-b_in^{-1}) \geq (1-n^{c_{d,\epsilon}})$, so 
\begin{equation*}
\geq (2i)!2^{-i-1}n^{i(i+3)}p^{i+1}(1-n^{-c_{d,\epsilon}}),
\end{equation*}  
completing the proof.
\end{proof}

 \begin{proposition}\label{prop.upperbound}
 
 Fix $d>2$ and $j\leq J_d$.  Let $f(n)\leq n^{d/{j+1} + j}$ satisfy $\lim_{n\to\infty} f(n)= \infty$ and $p = f(n)n^{-d/(j+1) - j},$ then $$ \prob_p(\inspn_{2j})\to 1.$$  
  
 \end{proposition}

\begin{proof}
   
First we define a sufficient event $E_{2j}\subset \inspn_{2j}$.  If we can show $\prob_p(E_{2j})\to 1$ then we can conclude $\prob_p(\inspn_{2j}) \to 1$ as well.     
 
For a fixed set of constants $\alpha = \{\alpha_{2j+1}, \cdots, \alpha_d \}$, let $V(\alpha)$ denote the subtorus given by $$V(\alpha) = \{ v\in [n]^d : v_i = \alpha_i \text{ for } 2j+1\leq i\leq d\}.$$  
 
There are $n^{d-2j}$ such subtori.  For $\alpha' = \{\alpha'_{2j+1},\cdots, \alpha'_d\}$, if $\alpha\neq \alpha',$ $V(\alpha)\cap V({\alpha'}) = \emptyset.$   Each event $\inspn_{V(\alpha)}$ will depend only on the nodes in $V(\alpha)$ so the events are independent.  The events will all have the same probability $\prob_p(\inspn_{V(\alpha)}) = \prob_p(\inspn_{V{(\alpha')}}).$  We now define the sufficient event, $$E_{2j} = \bigcup_{\alpha}\inspn_{V(\alpha)}.$$   We will show that $\prob_p(E_{2j}) \to 1$ for sufficiently large $p$ that satisfy the conditions of the proposition.  Since $E_{2j}\subset\inspn_{2j}$ this implies $\prob_p(\inspn_{2j}),\to 1$ as well.  

With this definition we have 

\beqlbl
\prob( E_{2j} ) = 1 - (1-M_{2j})^{n^{d-2j}} \geq 1- e^{-n^{d-2j}M_{2j}}.
\eeqlbl

We prove Proposition \ref{prop.upperbound} by proving that $n^{d-2j} M_{2j}\to \infty.$ 
  
First assume $f(n) < \log n$ and $j \leq J_d,$ $p = f(n)n^{-d/(j+1) -j}$ satisfies $p <  n^{-2J_d -\epsilon}$ for some $\epsilon>0$ and we may apply Lemma \ref{lemma.ealpha} to show 

\begin{align*}
n^{d-2j}M_{2j}  &\geq  n^{d-2j}(2j)!2^{-j-1}n^{j(j+3)}p^{j+1}(1-n^{-c_{d,\epsilon}}) \\
&\geq f(n)^{j+1}(1-n^{-c_{d,\epsilon}}) \to \infty.
\end{align*} 
\end{proof}

 If $\inspn_{2j}$ occurs then $\sublat_{2j}$ also occurs.  Proposition \ref{prop.upperbound} implies that for large enough $n$, $$p_c(2j-1,d)\leq p_c(2j,d)\leq p_\inspn(2j,d) < f(n)n^{-d/(j+1) - j}.$$ 

  The caveat that $f(n)< \log n$ is necessary only for the proof of the proposition.  Both $\prob_p(\inspn_{2j})$ and $\prob_p(\sublat_{2j})$ are increasing in $p$, so the proposition will still be true for faster growing $f(n)$ as long as $p\leq1$.

\subsection{ Lower Bound for $p_\inspn(i,d)$}

In this section we prove the lower bound for the critical exponent of $p_\inspn(2j,d).$  First let's start with the simplest possibilities for $V$: a single node, a line, and a plane.

\begin{itemize}
\item For a single node $u$, $$\prob_p( \inspn_{\{u\}} ) = p.$$
\item For a single line $L$,  $$\prob_p(\inspn_L) = \prob\big(\bin(n,p) \geq 2 \big) \leq {n \choose 2}p^2 = O(n^2p^2).$$ 
\item For a single plane $P$, $$\prob_p(\inspn_P) \leq \prob\big( \bin( n^2,p ) \geq 2 \big)\leq 2^{-1}n^4p^2.$$
\end{itemize}

Note that a plane is more likely to be internally spanned than a line because a line requires at least two collinear points.  The following lemma extends these computations.

\begin{lemma} \label{lemma.iupperbound} 

Fix $d$ and $j\leq J_d$ and let $p = f(n)n^{-d/(j+1) - j}$ for some $f(n) \to 0$.   For $1\leq i \leq j,$
\beqlbl \label{oddbound}
M_{2i-1} \leq O(n^{i(i+3)-2}p^{i+1}).
\eeqlbl
\beqlbl
M_{2i} = (1+ O(n^{-1}))(2i)!2^{-i-1}n^{i(i+3)}p^{i+1}. \label{evenbound}
\eeqlbl
Lastly $M_0 = p$.  
\end{lemma}

\begin{proof}{(By induction on $i$)}

We assume the lemma holds for all $1\leq l \leq  2i-2$ and show by induction that the formulas hold for dimensions $2i$ and $2i-1$.  Note the lemma holds for a line and a plane.  For a point we have $M_0 = p,$ which does not fit the formula and hence is mentioned separately.  

First let's assume a subtorus $V$ is internally spanned.  By Lemma \ref{lemma.inspnbreakdown}, there exists proper subtori $V_1$, $V_2 \subset V$ both internally spanned by disjoint non-empty subsets $S_1$ and $S_2$ such that $V = \spn{V_1,V_2}.$  Let $D_V$ denote the set of possible pairs of such subtori of $V$ with $\dim(V_1)\leq \dim(V_2)$.  $\inspn_V$ can be expressed as a union over $D_V$ of events of the form $\inspn_{V_1}\circ \inspn_{V_2},$ where $\circ$ denotes the disjoint occurrence of the two events.  By the union bound and the van den Berg-Kesten inequality \cite{grimmettpercolation} we have

$$\prob_p(\inspn_V)  \leq \sum_{ D_V} \prob_p( \inspn_{V_1}\circ \inspn_{V_2}) \leq \sum_{ D_V} \prob_p( \inspn_{V_1}) \prob_p( \inspn_{V_2}).$$

For $0 \leq t_1 \leq t_2 < \dim(V)$ let $D_V(t_1,t_2)$ denote the subset of $D_V$ where $\dim(V_1) = t_1$ and $\dim(V_2) = t_2$.  Since $\spn{V_1\cup V_2}$ is a subtorus it has dimension at most $t_1+t_2+2.$  Therefore if $t_1+t_2+2<\dim(V)$, then $D_V(t_1,t_2)$ is empty.  Otherwise $|D_V(t_1,t_2)|=O(n^{2i-t_1}n^{2i-t_2}).$  Then we have

\beqlbl  \label{dvsplit}
\sum_{D_V} \prob_p( \inspn_{V_1} )\prob_p( \inspn_{V_2} ) = \sum_{0\leq t_1\leq t_2}\sum_{D_V(t_1,t_2)} M_{t_1}M_{t_2} = \sum_{0\leq t_1\leq t_2} |D_V(t_1,t_2)|M_{t_1}M_{t_2}.
\eeqlbl

If $V\in \torus_{2i}$ we will show the probability $\inspn_V$ occurs is on the same order as the  probability there exists a pair $(V_1,V_2)\in D_V(0,2i-2)$ such that $\inspn_{V_1}\circ\inspn_{V_2}$ occurs.  

There exists a constant, $C$, depending only on $d$ such that 
$$|D_V(t_1,t_2)|M_{t_1}M_{t_2} \leq Cn^{4i-t_1-t_2}M_{t_1}M_{t_2}.$$

Note that $M_{2i-1} = O(n^{-2}M_{2i})$.  If $t_1= 2l-1$, then by the induction hypothesis $n^{2i-(2l-1)}M_{2l-1}=O( n^{-1}n^{2i-2l}M_{2l})$ so we may assume that $t_1$ (and $t_2$) are both even.  Let $t_1 =2i_1,$ and $t_2= 2i_2$, with $i_1+i_2 +1 = i + k$, where $0\leq k\leq i_1\leq i_2 < i $.  By the induction hypothesis we have an upper bound for $M_{2i_1}$ and $M_{2i_2}$.  

Therefore

\begin{align*}
|D_V(t_1,t_2)|M_{t_1}M_{t_2} \leq &Cn^{4i-2i_1-2i_2}M_{2i_1}M_{2i_2}\\
= & C(1+O(n^{-1}))^2n^{ 4i-2i_1-2i_2}n^{i_1(i_1 + 3) + i_2(i_2+3)}p^{i_1+1+i_2+1}\\
\leq& Cn^{-5i +k -1 + i_1^2+i_2^2}p^{i+1}\\
\leq &Cn^{i(i+3)}p^{i+1} n^{k(k-1) - 2i_1i_2}.
\end{align*}

If $i_1 >0$, then $k(k-1)-2i_1i_2 \leq -2.$  Therefore if $i_1 > 0$
\beqlbl \label{dab}
|D_V(t_1,t_2)|M_{t_1}M_{t_2}  = O(n^{-1})n^{i(i+3)}p^{i+1}.
\eeqlbl
If $t_1 = 0$ then $t_2= 2i-2$.  There are at most ${2i \choose 2}(n^{2i}n^2)$ pairs in $D_V(0,2i-2)$.  Therefore
 
 \beqlbl \label{d02i2}
|D_V(0,2i-2)|M_0M_{2i-2} \leq {2i-2 \choose 2}(n^{2i}n^2)(2i-2)\!2^{-i}n^{(i-1)(i+2)}p^{i}p(1+O(n^{-1}))
 \eeqlbl
Combining Equations \eqref{dab} and \eqref{d02i2}
 \beqlbl
 \sum_{D_V} \prob_p(\inspn_{V_1}\circ \inspn_{V_2} ) \leq (1+O(n^{-1})){2i \choose 2}(n^{2i}n^2)n^{i(i+3)}p^{i+1}
 \eeqlbl gives an upper bound for $M_{2i}.$  This inequality combines with Lemma \ref{lemma.ealpha} to prove Equation \eqref{evenbound} of Lemma \ref{lemma.iupperbound}.
 
 A similar argument shows that for $\dim(V) = 2i-1$ the sum is dominated by the terms from $D_V(0,2i).$  If $\dim(V) = 2i-1,$ then there are at most $O(n^{2i-1}n)$ pairs in $D_V(0,2i)$.  The union bound gives

$$|D_V(0,2i)| M_{0}M_{2i-2} = O( n^{2i-2}n ) M_{0}M_{2i-2} = O\left(n^{(i(i+3)-2}p^{i+1}\right),$$ proving \eqref{oddbound} of Lemma \ref{lemma.iupperbound}.
\end{proof}

\begin{proposition} \label{prop.lowerbound}

Fix $d$ and $j\leq J_d$.  For any $f(n) \to 0$, if $p = f(n)n^{-d/(j+1) - j},$ then 

$$\prob_p(\inspn_{2j}) \to 0.$$   
 
\end{proposition}

This proposition implies $p_\inspn(2j,d) > f(n)n^{-d/(j+1) - j}.$  Unlike Proposition \ref{prop.upperbound}, we need a little extra care to claim $p_c(2j,d) > f(n)n^{-d/(j+1)-j}$ (see Section \ref{pcbound}).  
 
\begin{proof}

The union bound gives: $$\prob_p(\inspn_{2j}) \leq \sum_{ V\in \torus_{2j} } \prob_p( \inspn_V)\leq {d\choose 2j}n^{d-2j}M_{2j}.$$ 

By Lemma \ref{lemma.iupperbound}, $M_{2j} = O(f(n)^{j+1} n^{2j-d} )$ when $p = f(n)n^{-d/(j+1)-j}.$  Then $\prob_p( \inspn_{2j} ) = O(f(n)^{j+1}) \to 0$ which implies $p_\inspn(2j,d) >  f(n)n^{-d/(j+1)-j}.$
\end{proof}

%

%
%
%

\subsection{ Bounds for $p_c(2j,d)$ } \label{pcbound}

In this section we will show $\prob_p(\sublat_{2j} \backslash \inspn_{2j} ) \to 0.$   We will show that if $\prob( \inspn_{2j} )= 0$ then $\prob_p( \sublat_{2j} ) = 0.$  By Proposition \ref{prop.lowerbound} we have for fixed $d$ and $j\leq J_d$ with $f(n) \to 0$, $$p_\inspn(2j,d) \geq f(n) n^{-d/(j+1)-j}$$ for large enough $n.$

If $\sublat_{2j}$ occurs then there exists some subtorus with dimension greater than or equal to $2j$ that is internally spanned.  The next lemma will show that for any dimension $b>2j$, $\prob_p(\inspn_b)\to 0$ if $\prob_p(\inspn_{2j})\to 0.$  This implies that $\prob_p(\sublat_{2j})\to 0$ as well.

\begin{lemma}\label{nob}
Fix $d$ and $j \leq J_d$, and let $p=an^{-d/(j+1)-j}.$  For $b> x$ and $V\in \torus_b$, let $\binspn^x_V$ denote the event that $V$ is internally spanned and no subtorus contained in $V$ with dimension exactly $x$ is internally spanned.  Let $\binspn^{x}_b = \cup_{V\in \torus_b} \binspn^x_V.$  Then,
\beq
\prob_p(\binspn^{2j}_b) \to 0.
\eeq
\end{lemma}

\begin{proof}

By Lemma \ref{lemma.inspnbreakdown}, if $\inspn_{V}$ occurs for some $V \in \torus_b$, there exist $V_1$ and $V_2 \subset V$ with $\dim(V_1)\leq \dim(V_2) < b$ such that $\inspn_{V_1}\circ\inspn_{V_2}$ occurs and $\spn{V_1\cup V_2} = V$.  If $\dim(V_2)>2j$ we may repeatedly apply Lemma \ref{lemma.inspnbreakdown} until we have a pair of subtori $(V_1',V_2')$ such that $\dim(V'_1) \leq \dim(V'_2) < 2j$, $\inspn_{V'_1}\circ \inspn_{V'_2}$ occurs and $V' = \spn{V'_1\cup V'_2}$ with $\dim(V')=b'>2j$.  If $\dim(V_1) =0$ then $\dim(V_2) = 2j-1$.  By Lemma \ref{lemma.iupperbound} and the union bound,
\beq
\prob_p(\inspn_{2j-1}) \leq \bigo{n^{d-2j+1}M_{2j-1} }\leq \bigo{n^{d-2j+1}n^{2j-2-d}} = o(1).
\eeq
Therefore we may assume $0< \dim(V'_1)\leq \dim(V'_2) < 2j-1.$

Let $T$ denote the set of $t_1,t_2$ such that $1\leq t_1\leq t_2 < 2j$ and $t_1+t_2 \geq b-2.$  We will assume for simplicity $t_1=2i_1, t_2 = 2i_2$ and $b = 2i = 2j + 2k$ for some $0<k<i_1\leq i_2 <j.$  The computations where $t_1,t_2$ or $b$ are odd follow similar arguments as those that follow.

Let $T_0$ denote the subset of $T$ such that $t_1+t_2 = b-2.$  The expression $n^{2b-t_1-t_2}M_{t_1}M_{t_2}$ decreases if $t_1$ or $t_2$ increases.  For each $(t_1,t_2)\in T_0$ such that $t_1+t_2 = b-2$ there are at most $4j$ pairs $(x_1,x_2)\in T$ where $x_1\geq t_1$ and $x_2\geq t_2$.  For any $V\in \torus_b$

\begin{align*}
\prob_p(\binspn^{2j}_V)&= O\left( \sum_{(x_1,x_2) \in T}n^{2b-x_1-x_2} M_{x_1}M_{x_2} \right)\\
&= O\left( \sum_{(t_1,t_2)\in T_0 } 4jn^{2b - t_1 - t_2 }M_{t_1}M_{t_2} \right)\\
& = O\left( \sum_{(t_1,t_2)\in T_0}n^{4i-2i_1-2i_2 +i_1^2+i_2^2 + 3i_1+3i_2  }p^{i_1+i_2 + 2} \right )\\
&=O\left( \sum_{(t_1,t_2)\in T_0}n^{-2i_12i_2 + j^2+ k^2 + 2jk+3j+3k + 1} n^{-d-j(j+1)} p^{k}\right)\\
& = O\left( \sum_{(t_1,t_2)\in T_0}n^{2j+2k-d +1-2i_1i_2 + k(k + 1  ) }\right)\\
& = O\left( \sum_{(t_1,t_2)\in T_0}n^{b - d - 1}\right)\\
& = O\left( n^{b-d-1}\right )
\end{align*}
since $1-2i_1i_2+k(k+1) \leq -1.$   

There are only $O(n^{d - b})$ subtori in $\torus_b$ so $\prob_p(\binspn_b) = \bigo{n^{d-b}n^{b-d-1}} = o(1).$  
\end{proof}
%
%

\begin{corollary}\label{lemma.noinspnnoexist}

Fix $d$ and $j \leq J_d$, and let $p = an^{-d/(j+1) -j}.$  Then 
$$\prob_p( \subtorusexists_{2j}\backslash \inspn_{2j} ) \to 0.$$ 
\end{corollary}

\begin{proof}[Proof of Lemma \ref{lemma.noinspnnoexist}]

If $\subtorusexists_{2j}$ occurs, then by Lemma \ref{lemma.maxinspn} there must be some $s$-dimensional subtorus $V$ such that $\inspn_{V}$ occurs and $s\geq 2j$.  Let $b$ be the minimal such $s$ and suppose $b > 2j.$  

 Therefore the probability there exists an internally spanned subtorus of dimension greater than $2j$ tends to zero if no subtorus of dimension $2j$ is also internally spanned.
\end{proof}

Now we can conclude that $p_c(2j,d)$ is also bounded below $f(n)n^{-d/(j+1) -j}$ for any $f(n)\to 0$.

\section{Poisson Approximation}

We use the Chen-Stein method for approximation by a Poisson distribution for postively related random variables. 
\begin{theorem}[Ross \cite{ross.steinmethod}, 4.14]

Let $X_1, \dots, X_m$ be indicator variables with $\prob( X_i = 1) = p_i,$ $Y=\sum_{i=1}^m X_i$, and $\lambda = \expect[Y] = \sum_i^mp_i.$  For each $i\in [m],$ let $\left(X_j^{(i)}\right)$ have the distribution of $(X_j)_{j\neq i }$ conditional on $X_i = 1$ and let $I$ be a random variable independent of all else, such that $\prob( I = i) = p_i/\lambda$ so that $Y^s = \sum_{j\neq I} X_j^{(I)} + 1$ has the size-bias distribution of $Y$.  If $X^{i}_j\geq X_j$ for all $i\neq j$ and $Z \sim \text{Po}(\lambda),$ then  


\beqlbl \label{eq.steinchen}
d_{TV}( Y,Z ) \leq \left ( \var(W) - \lambda + 2\sum_{i=1}^m p_i^2 \right ).
\eeqlbl
\end{theorem}

\begin{proof}[Proof of Theorem \ref{theorem3}]

Let $X_V$ denote the indicator random variable for the event $\inspn_V$.  Furthermore, for $W\in \torus_{2j}$ let $X_V^W$ denote the indicator function for the event $\inspn_V$ conditioned on $X_W = 1$.  If $V \cap W =\emptyset $ then $X_V^W = X_V.$  Otherwise $X_V^W \geq X_V.$    

For all $V,W\in\torus_{2j}$, 

\beq p_V = p_W  = M_{2j} = (2j)!2^{-j-1}a^{j+1}n^{2j-d}(1+ o(1))
\eeq and 
\beq p_{VW} = \expect[X_V X_W] = \prob_p( \inspn_V \cap \inspn_W ).
\eeq  Let $Y = \sum_{\torus_{2j}} X_V$.  Then 

\beqlbl
\lambda = \expect[Y] = (1+ o(1))\sum_{\torus_{2j}} (2j)!2^{-j-1}a^{j+1}n^{2j-d} = (1+o(1)){d \choose 2j}(2j)!2^{-j-1}a^{j+1}. \label{eq.lambda}
\eeqlbl  
If $V\cap W = \emptyset$ then $p_{VW} = p_Vp_W$ so will contribute nothing $\var(Y)$.  Let $\Gamma_V$ denote subset of $\torus_{2j} \backslash V$ such that $W\notin \Gamma_V$ implies $W\cap V = \emptyset$ or $W = V.$  Then 
\beq 
\var(Y) \leq\sum_{V\in \torus_{2j}}\left(  p_V +  \sum_{W \in \Gamma_V} p_{VW}\right).
\eeq
Finally we let $Z\sim \text{Po}(\lambda),$ a Poisson random variable with parameter $\lambda.$\\

Using Inequality \ref{eq.steinchen} we get

\beqlbl \label{eq.stch}
d_{TV}( Y, Z ) \leq \min\{1,\lambda^{-1}\}\sum_{V\in\torus_{2j}}\left ( p_V + \sum_{W\in \Gamma_V} p_Vp_W \right ) -\lambda + 2 \sum_{V \in \torus_{2j}} p_V^2.
\eeqlbl  Immediately we see that $\sum_{\torus_{2j}} p_V = \lambda$ so we can simplify Inequality \ref{eq.stch} to 

 \beqlbl \label{simplestein}
  d_{TV}( Y, Z ) \leq \sum_{V\in\torus_{2j}}\sum_{W\in \Gamma_V} p_{VW} + 2\sum_{V\in \torus_{2j}} p_V^2.
 \eeqlbl
 
The second part of the right hand side of \ref{simplestein} is easiest to deal with.  The size of $\torus_{2j}$ is $O(n^{d-2j})$, while $p_V = O(n^{2j-d}).$  Hence $2|\torus_{2j}|M_{2j}^2 = O(n^{d-2j} )O(n^{2j-d})^2 = O(n^{2j-d}) \to 0.$

For the first part of the right hand side of \ref{simplestein} will require a little more care.  For $0\leq r \leq 2j-1$ let $\Gamma^r_V$ denote the subset of $\Gamma_V$ such that for $W \in \Gamma^r_V$, $\dim(V \cap W) = r.$  

For a fixed $V \in \torus_{2j}$ let $\torus^V$ denote the set of subtori contained in $V$.  Let $\torus^V_r = \torus_{r} \cap \torus^V.$  For $U \in \torus^V_r$, let $\inspn_{U\to V}$ denote the even that $V$ is internally spanned conditioned on $U$ being completely open.  We state two lemmas whose proofs are rather technical and delayed until the appendix.

\begin{lemma} \label{condprob}

Fix $d>2$, and let $j \leq J_d$ and for some $\epsilon>0$, let $p \leq n^{-2j-\epsilon}.$  For $r < t\leq 2j$, let $x = \ceil{(t-r)/2}$.  Define $$f(t,r,x) = (x-1)(x+r) +t.$$  For $U \in \torus_{r}^V,$ 

\beqlbl
\prob_p(\inspn_{U\to V}) = \bigo{n^{ f(t,r,x)}p^{x}}.
\eeqlbl


\end{lemma}

\begin{lemma} \label{intersection}

Fix $d>2$, and let $j\leq J_d$ and $p \leq n^{-2j-\epsilon}$.  Fix $r< s\leq t\leq 2j$ and suppose $V\in \torus_t, W\in \torus_s, $ and $ V\cap W\in \torus_r.$  Let $i = \ceil{t/2},$ $k = \ceil{s/2}$, $l= \floor{r/2}$, $x = \ceil{(t-r)/2},$ and $y = \ceil{(s-r)/2}.$  Let $f$ be defined as in Lemma \ref{condprob}.  Then

\beqlbl \label{eq.int}
p_{VW} = \bigo{n^{l^2 + l+r}n^{f(t,2l,i-l)}n^{f(s,2l,k-l)}p^{i+k-l+1}}. 
\eeqlbl

\end{lemma}

Assuming the lemmas are true we can finish the proof of Theorem \ref{theorem3} rather easily.  Let $t = s= 2j$, $r < 2j.$  For $V\in \torus_{2j}$ the size of $\Gamma_V^r$ is $\bigo{n^{t-r}}.$  For $p=an^{-d/(j+1) - j}$, if $\epsilon < \frac{ 1}{j+1}$ then for large $n$, $p < n^{-2j - \epsilon}.$  Then
 
\begin{align*}
|\Gamma_V^r|p_{VW} &= \bigo{n^{2j-r + l^2+l+r}n^{2f(2j,2l,j-l)}p^{2j-l+1}}\\
&= \bigo{n^{2j-d-\epsilon}}.
\end{align*}



Therefore $$
\sum_{V\in\torus_{2j}} \sum_{r=0}^{2j} \sum_{W\in \Gamma^r_V} p_{VW}= \bigo{n^{d-2j}}\bigo{n^{2j-d-\epsilon}} = o(1).$$
\end{proof}

\section{Proofs of Theorems}

Theorem \ref{theorem1} can viewed as an immediate corollary of Theorem \ref{theorem3} and Lemma \ref{lemma.noinspnnoexist}.   These combine to show $\prob_p(\sublat_{2j} \backslash \inspn_{2j} ) \to 0.$

\begin{proof}[Proof of Theorem \ref{theorem4}]

We will use "sprinkling" as in \cite{bbmhighdim} to show that if $\inspn_{2J_d}$ occurs $\inspn_t$ occurs for $t\geq {2J_d+2}.$  If $d<(J_d+1)(J_d+2)$ then for some $\epsilon >0,$ $d/(J_d+1) + J_d < 2J_d +2 - \epsilon.$  For $\delta > 0 $ let $p = (a+\delta)n^{-d/(J_d+1) -J_d}$, $p_1 = an^{-d/(J_d+1)-J_d}$, and $p_2= n^{-2J_d-2 + \epsilon}.$  For large enough $n$, $p_1+p_2 < p$.

Consider two random initial configurations $\omega_0^1$ and $\omega_0^2$ where each node in $[n]^d$ is open with probability $p_1$ and $p_2$ respectively.  For large enough $n$ the union of these configurations, $\omega_0^1\cup \omega_0^2,$ is stochastically dominated by the random configuration, $\omega_0,$ where each node is open with probability $p$.  

For each $V \in \torus_{2J_d}$ let $V\subset V_{2}\subset \cdots \subset V_{d-2J_d}$ be a seqeunce of subtori such that $V_k \in \torus_{k-2J_d}$ for $2\leq k \leq d-2J_d$.  Furthermore, let $N_{k}(V)$ denote the set of nodes in $V_k$ that are exactly distance $2k$ away from $V$.  The size of $N_k(V)$ is at least $cn^{2J_d+k}$ for each $k$.  Let $\nbr_V^k$ denote the even that for each $2\leq i\leq k$, $N_i(V)$ contains at least 1 open node.  If $V$ is internally spanned and $\nbr_V^k$ occurs, then $V_k$ is internally spanned.  

There are at least $\frac12n^{2J_d+2}$ nodes in each $N_{i}(V),$ so 
\beqlbl \label{sprink}
\prob_{p_2} (\nbr^k_V) = \prod_{i=2}^k\left(1-( 1 - p_2)^{ |N_i(V)|}\right) = 1-o(1).
\eeqlbl

Fix an ordering of $\torus_{2J_d}$ and let $\inspn'_V$ denote the even that $V$ is the first subtorus in the ordering such that $V$ is internally spanned.  The event $\inspn_{2J_d}$ is the disjoint union $\cup_{V} \inspn'_V.$  For $2\leq k \leq d-2J_d$,
\begin{align*}
\prob_p(\inspn_{2J_d} \cap \inspn_{2J_d+k} ) &\geq \sum_{V \in \torus_{i}} \prob_{p}\left(\inspn'_V\cap \nbr^k_V\right)\\
& \geq \sum_{V} \prob_{p_1,p_2}( \inspn'_V(\omega_0^1)\cap\nbr^k_V(\omega_0^2))\\
& \geq \sum_{V} \prob_{p_1}(\inspn'_V)\prob_{p_2}(\nbr^k_V) \\
& = \sum_V \prob_{p_1}(\inspn'_V)(1-o(1))\\
&= \prob_{p_1}(\inspn_{2J_d})(1-o(1)).
\end{align*}
For any $\delta>0,$
\begin{align*}
 \lim \sup \prob_p(\inspn_{2J_d}\backslash \inspn_{2J_d+k}) &\leq  \lim \sup ( \prob_p(\inspn_{2J_d}) - \prob_{p_1}(\inspn_{2J_d}) )\\
& = e^{-\lambda(J_d,d,a)}- e^{-\lambda({J_d},d,a+\delta)}.
\end{align*}
This last expression tends to 0 with $\delta$, concluding the proof.  
\end{proof}


\begin{proof}[Proof of Theorem \ref{theorem5}]

If $(J_d+1)(J_d+2)=d\geq6$ then $$d/(J_d+1) + J_d = d/(J_d+2) + (J_d+1) = 2J_d+2.$$ Unlike in Theorem \ref{theorem4} we do not necessarily have unstoppable growth once we have at least one internally spanned subtorus of dimension $2J_d$.  Equations \eqref{thm51} and \eqref{thm53} state that with positive probability there is unstoppable growth while Equation \eqref{thm52} says the internally spanning a $2J_d$-dimensional subtorus does not guarantee the spanning of a $2J_d+2$ dimensional subtorus.

To prove \eqref{thm51} and \eqref{thm52} we will modify the sprinkling arguments from the proof of Theorem \ref{theorem4}.  Let $p = an^{-2J_d-2}$, $p_1 = p_2 = \frac a2n^{-2J_d-2}.$  

The difference with previous arguments is that for Equation \eqref{sprink}, we have instead for some $c>0,$
$$\prob_{p_2}( \nbr^2_V ) \geq\left ( 1-\left (1-\frac a2n^{-2J_d-2}\right)^{cn^{2J_d+2}} \right)(1-o(1)) = (1- e^{-ac/2})(1-o(1)).$$
Repeating previous arguments
\begin{align*}
\prob_p(\inspn_{2J_d} \cap \inspn_{2J_d+2} ) &\geq  \sum_{V} \prob_{p_1}(\inspn'_V)\prob_{p_2}(\nbr^k_V) \\
& = \sum_V \prob_{p_1}(\inspn'_V)\left(1-(1-p_2)^{cn^{2J_d+2}}\right)\\
&= \prob_{p_1}(\inspn_{2J_d})(1-e^{-ac/2})(1-o(1)).
\end{align*}
Therefore $$\prob_p(\inspn_{2J_d+2}) \geq (1-e^{\lambda(J_d,d,a/2)})(1-e^{-ac/2})(1-o(1)) > 0,$$ proving \eqref{thm51}.

For \eqref{thm52}, we again let $p = an^{-2J_d-2},$ and $Y_{2J_d}$ denote the number of subtori of dimension $2J_d$ that are internally spanned.

We may view the even $\inspn_{2J_d}$ as the disjoint union $\cup_{k=1}^{\infty}\{Y_{2J_d} = k\}$.  By Theorem \ref{theorem3} we know that $$\prob_p(Y_{2J_d} = 1) = e^{-\lambda(J_d,d,a)}\lambda(J_d,d,a)(1+o(1)) >0.$$  Let $\inspn^*_V$ denote the event that $V$ is internally spanned and no other $V' \in \torus_{2J_d}$ is internally spanned.  Then $$\prob_p(Y_{2J_d} = 1) = \prob_p\left(\bigcup_{V\in\torus_{2J_d}} \inspn^*_V\right).$$

For $V\in \torus_{2J_d}$ let $Q(V)$ denote the event that every node exactly distance 1 or 2 away from $V$ is not open.  For some $C>0$, there are at most $Cn^{2J_d+2}$ such nodes.  All are not open with probability at least $(1-p)^{Cn^{2J_d+2}} = e^{-aC}(1-o(1)).$  Moreover $Q(V)$ and $\inspn^*_V$ are positively related, so 
$$\prob_p(\inspn^*_V \cap Q(V) ) \geq \prob_p(\inspn^*_V)\prob_p(Q(V)).$$

For $W\in \torus_{2J_d+2}$, recall $\binspn_{W}$ denotes the event that $W$ is internally spanned but no $W'\subset W$ with $W'\in \torus_{2J_d}$ is internally spanned.  By Lemma \ref{nob} $\prob_p(\binspn_{2J_d+2}) =o(1).$ 

\begin{align*}
\prob_p\left( \inspn_{2J_d} \cap \left(\inspn_{2J_d+2}\right)^c \right) &\geq \prob_p\left( \{Y_{2J_d} =1\} \cap \left(\inspn_{2J_d+2}\right)^c \right)\\
&\geq \prob_p\left( \bigcup_{V\in \torus_{2J_d}} \inspn^*_V \cap Q(V) \right ) - \prob_p( \binspn_{2J_d+2} )\\
&\geq\left( \sum_{V\in \torus_{2J_d}} \prob_p(\inspn_V^*)\prob_p(Q(V) )\right)- o(1)\\
&\geq  \left(\sum_V \prob_p(\inspn^*_V)e^{-ac}(1-o(1))\right) - o(1)\\
&\geq e^{-ac}\prob_p(\{Y_{2j} = 1\})(1-o(1)) - o(1)\\
&\geq e^{-ac}\lambda(J_d,d,a) e^{-\lambda(J_d,d,a)}(1-o(1)).
\end{align*}
The last line is positive for large enough $n$.

For each $V\in \torus_{2J_d+2}$, and $0<k\leq d-2J_d-2$, $\prob_p(\nbr_V^k) \geq (1-e^{-n/2})^k.$  There are at most $n^d$ subtori in $\torus_{2J_d+2}$.  Let $R = \bigcap_{V \in \torus_{2J_d+2}} \nbr_V^k.$  The event $\{\inspn_{2J_d+2}\cap R \}$ is a subset of $\{\inspn_{2J_d+2}\cap \inspn_{2J_d+2+k}\} $.  By a very crude union bound $\prob_p(R^c) \leq kn^de^{-n/2}.$  Then 

\begin{align*}
\prob_p( \inspn_{2J_d+2} ) &= \prob_p( \inspn_{2J_d+2} \cap R ) + \prob_p( \inspn_{2J_d+2} \cap R^c )\\
& \leq \prob_p( \inspn_{2J_d+2} \cap \inspn_{2J_d+2+k}) + kn^de^{-n/2}\\
&\leq \prob_p( \inspn_{2J_d+2} ) + kn^{d}e^{-n/2}
\end{align*}

For large $n$ we see that $\prob_p( \inspn_{2J_d+2+k} )\to \prob_p( \inspn_{2J_d+2}).$  Letting $k = d-2J_d-2$ proves \eqref{thm53}.
\end{proof}

\begin{proof}[Proof of Theorem \ref{theorem6}]

The first part of the theorem will follow from arguments similar to Lemma \ref{nob}.  We will show for $d=(J_d+1)(J_d=2)$ and $J_d>1$, if $p=an^{-2J_d-2}$ $b>2J_d+2$, then

$$\prob_p( \binspn^{2J_d+2}_b ) \to 0.$$

As in Lemma \ref{nob}, if $\binspn_{b}^{2J_d+2}$ occurs then for $2J_d+2<b'\leq b$, there exists $W \in \torus_{b'}$ and $W_1,W_2\subset W$ such that $\dim(W_1) \leq \dim( W_2) < 2J_d+2$, $\spn{W_1,W_2} = W$ and $\inspn_{W_1} \circ \inspn_{W_2}$ occurs.  

Let $T$ denote all pairs of $(x_1,x_2)$ such that $x_1\leq x_2 < b'$ and $x_1+x_2\geq b'-2.$  Let $T_0$ denote the subset of pairs $(t_1,t_2)\in T$ such that $t_1 + t_2 = b'-2.$  The computations where $t_1,t_2$ or $b'$ are odd follow similar arguments as those that follow.

For simplicity let us assume that $\dim(W_1) = t_1 =  2i_1$, $\dim(W_2) = t_2 =  2i_2$, $b' = 2J_d+2 + 2k$ and $i_1 + i_2 = J_d+k.$  There are $\bigo{n^{2b'-t_1-t_2}}$ choices of $W_1\in\torus_{t_1}$ and $W_2\in \torus_{t_2}.$  \begin{align*}
\prob_p( \binspn^{2J_d+2}_{W} ) &= \bigo{ \sum_{(x_1,x_2)\in T} n^{2b'-x_1-x_2}M_{x_1}M_{x_1} }\\
&=\bigo{\sum_{(t_1,t_2)\in T_0} (4J_d+4)n^{2b'-t_1-t_2}M_{t_1}M_{t_2}}\\
&= \bigo{ \sum_{(t_1,t_2)\in T_0} n^{4J_d+4k+4 - 2i_1-2i_2 +  i_1^2+i_2^2 + 3i_2+3i_2}p^{i_1+i_2+2}}\\
&= \bigo{\sum_{(t_1,t_2)\in T_0}n^{2J_d+2k+2 -d+k^2+k - 2i_1i_2}}\\
&= \bigo{\sum_{(t_1,t_2) \in T_0} n^{b'-d + k^2+k-2i_1i_2}}.
\end{align*}
Here we diverge slightly with the proof of Lemma \ref{nob}.  When $d=6$, if $i_1 = i_2 = k = J_d=1,$ then $k^2 + k - 2i_1i_2 =0$ which will cause issues.  However, for $d>6$ it must be that $i_2\geq 2.$  In this case we have $k^2+k \leq 2i_1i_2 -1$ so the above bounds give

$$\prob_p(\binspn^{2J_d+2}_{b'} )= \bigo{\sum_{W \in \torus_{b'}} \prob_p( \binspn^{2J_d+2}_{W})} = \bigo{n^{d-b'}n^{b'-d-1}}= o(1).$$

In particular, this says $$\prob_p(\inspn_d\backslash\inspn_{2J_d+2}) \leq \sum_{2J_d+2<b\leq d}\prob_p(\binspn^{2J_d+2}_b)  = o(1),$$ proving \eqref{thm62} Theorem \ref{theorem6}.  

Lastly we will prove \eqref{thm63} of Theorem \ref{theorem6}, there exists $c>0$ such that for large enough $n$, $$\prob_p(\inspn_6\backslash \inspn_4)> c.$$  

Let $$\vleft = \{( *, *, a_3, a_4,a_5,a_6) \ | \ 1\leq a_3,a_4,a_5,a_6\leq n/2 \}$$ and $$\vright = \{ ( b_1,b_2,b_3,b_4,*,*) \ | \ n/2 < b_1,b_2,b_3,b_4 < n \}.$$

For a pair $(V,W)$ such that $V\in \vleft$ and $W\in \vright$, if $\inspn_V\cap \inspn_W$ occurs then $\inspn_d$ also occurs.

For each $V\in \torus_2$ and some constants $0<c_1<c_2$, such that $c_1n^{-4} \leq  \prob_p( \inspn_V ) \leq c_2n^{-4}.$  Again let $Q(V)$ denote the event that all nodes with distance exactly 1 or 2 from $V$ are note open.  There are at most ${6 \choose 2}n^4$ possible nodes, so for large $n$, $\prob_p(Q(V)) \geq (1- an^{-4})^{15n^4} \geq e^{-30a}.$ 

For a pair $(V,W)$ such that $V \in \vleft$ and $W\in \vright$.  Let $E(V,W)$ denote the event that all subtori in $\torus_2$ except for possibly $V$ and $W$ are not internally spanned.  There are at most $15n^4-2$ such subtori each with probability at least $1-c_2n^{-4}$ of not spanning.  Then $$\prob_p(E(V,W)) \geq (1-c_2n^{-4})^{15n^4} \geq e^{-15c_2} >0.$$
The event $E(V,W)$ is positively related to both events $Q(V)$ and $Q(W)$, so $$\prob_p(Q(W)\cap Q(V) \cap E(V,W)) \geq \prob_p(Q(V))\prob_p(Q(2))\prob_p(E(V,W).$$  Furthermore, the events $\inspn_V, \inspn_W, Q(V),$ $Q(W)$ are all pairwise independent.  Therefore

\begin{align*}
\prob_p(\inspn_6\backslash \inspn_4) &\geq \prob_p\left(  \bigcup_{V\in \vleft, W\in \vright} \inspn_V \cap \inspn_W \cap Q(V) \cap Q(W) \cap E(V,W) \right ) \\
&\geq\sum_{V\in \vleft, W\in \vright} \prob_p\left( \inspn_V \cap \inspn_W \cap Q(V) \cap Q(W) \cap E(V,W) \right )\\
&\geq\sum_{V\in \vleft, W\in \vright}  \prob_p(\inspn_V) \prob_p(\inspn_W)\prob_p(Q(V))\prob_p(Q(W))\prob_p(E(V,W))\\
&\geq \sum_{V\in \vleft, W\in \vright} (c_1n^{-4})^2e^{-60a}e^{-15c_2}\\
&\geq \left( \lfloor n/2 \rfloor \right)^{4}\left( \lfloor n/2 \rfloor \right)^{4}n^{-8}c_1^2e^{-60a-15c_2}\\
& \geq {c_1^2}e^{-60a-15c_2}/256\\
&>0.
\end{align*}
\end{proof}

\section{Acknowledgements}

The author would like to thank Christopher Hoffman and Janko Gravner for invaluable mentorship throughout this project.  The author would also like to thank the referees for very helpful critiques and suggestions for improvement.

\section{Appendix}

In this appendix we provide proofs for Lemmas \ref{condprob} and \ref{intersection}

\begin{proof}[Proof of Lemma \ref{condprob} (by induction on $t$)]

In any inductive proof one must establish a base case.  Recall that $\inspn_{U\to V}$ denotes the event that $V$ is internally spanned conditioned on the event that $U$ is completely open.  Also recall $\dim(V) = t$ and $\dim(U) = r$ with $U\subset V$.  If $t-r=\leq 2$ then $\inspn_{U\to V}$ occurs if any generic point in $V$ is open so $\prob_p(\inspn_{U\to V}) = n^{t}p$ in this case.  

In order to understand how $\inspn_{U\to V}$ occurs we recall Lemma \ref{lemma.inspnbreakdown}.  If $V$ is internally spanned then there are two subtori $V_1$ and $V_2$ that are disjointly internally spanned and $\spn{V_1\cup V_2 } = V$ and both $\dim(V_1)$ and $\dim(V_2)$ are less than $t$.  Through the same arguments of Lemma \ref{lemma.inspnbreakdown} if $\inspn_{U\to V}$ occurs then there are two subtori of dimension less than $t$, $V_1$ and $V_2$ such that $V=\spn{V_1\cup V_2}$ and $\inspn_{U\to V_1}$ and $\inspn_{U\to V_2}$ occur disjointly.

If $V_1\cap U =\emptyset$ then $\inspn_{U\to V_1}$ occurs if and only if $\inspn_{V_1}$ occurs.  Otherwise if $V_1\cap U \neq \emptyset$ let $W = \spn{V_1\cup U}.$  If $\inspn_{U\to V_1}$ occurs, then $\inspn_{U\to W}$ also occurs.  Therefore we may assume that $V_1\cap U =\emptyset $ or $U$.  We may also assume that $V_2$ does not intersect $U$.  

Let $D^0_V$ denote the subset of $D_V$ such that $V_1$ does not intersect $U$, and let $D^1_V$ the subset of $D_V$ where $V_1$ contains $U.$  

Combining the two cases gives the following upper bound 

\beqlbl 
\prob_p(\inspn_{U\to V}) \leq \sum_{k=0}^1\sum_{D^k_V}\prob_p(\inspn_{U\to V_1}) \prob_p(\inspn_{U\to V_2}).
\eeqlbl

Let $Q=Q(t,r,x) = n^{f(t,r,x)}p^x$.  For each $k=0,1$ we will show $$\frac{1}{Q}\sum_{D^k_V}\prob_p(\inspn_{U\to V_1}) \prob_p(\inspn_{U\to V_2}) = \bigo{1}.$$  

Before proceeding with the proof we provide a list of definitions of variables that we will use.

\begin{itemize}
\item $t = \dim(V)$, $r= \dim(U)$, $i = \ceil{t/2}$, $x= \ceil{(t-r)/2}$, $\alpha = 2i-t$, $\beta=2x-t+r$.   
\item $t_1 = \dim(V_1)$,  $i_1 = \ceil{t_1/2}$, $x_1= \ceil{(t_1-r)/2}$, $\alpha_1 = 2i_1-t_1$, $\beta_1=2x_1-t_1+r$,
\item $t_2 = \dim(V_2)$,  $i_2 = \ceil{t_2/2}$, $x_2= \ceil{(t_2-r)/2}$, $\alpha_2 = 2i_2-t_2$, $\beta_2=2x_2-t_2+r$.   
\end{itemize}

The variables $\alpha, \alpha_1, \alpha_2$ and $\beta, \beta_1, \beta_2$ all are in $\{0,1\}$ depending on the parity of $t,t_1, t_2$ and $t-r,t_1-r,t_2-r.$

For $(V_1,V_2) \in D^0_V$

$$\prob_p(\inspn_{U \to V_1})\prob_p(\inspn_{U\to V_2}) = \prob_p(\inspn_{V_1}) \prob_p(\inspn_{V_2})$$
so 
$$\sum_{D^0_V} \prob_p(\inspn_{U\to V_1}\circ\inspn_{U\to V_2}) \leq \sum_{D_V} \prob_p(\inspn_{V_1}) \prob_p(\inspn_{V_2}) =\bigo{\prob_p(\inspn_V)} = \bigo{n^{i^2+2t-i}p^{i+1}}.$$

The exponent of $n^{i^2+2t-i-f(t,r,x)}p^{i+1-x}$ is at most

$$i^2+2t-i-(x-1)(x+r) - t - 2j(i+1-x) -\epsilon(i+1-x).$$

Rearranging the terms and noting that $1\leq x \leq i \leq j$ and $t\leq 2j$ is apparent that 

$$(i+x- 2j)(i-x)+ (t-2j) + r(1-x) + (x-i)-\epsilon(i-x+1)$$

is at most $-\epsilon$ and therefore 

\beqlbl \label{d0vbound}
\sum_{D^0_V} \prob_p(\inspn_{U\to V_1}\circ\inspn_{U\to V_2})=\bigo{Q(t,r,x)}.
\eeqlbl

Next we consider the contribution from $D^1_V$.  Let $D^1(t_1,t_2)$ denote the subset of $D^1_V$ such that $\dim(V_1) = t_1$ and $\dim(V_2) = t_2.$  There are $\bigo{1}$ possibilities for $V_1$ and $\bigo{n^{t-t_2}}$ possibilities for $V_2$.  Therefore for each $(t_1,t_2)$ that satisfies $t_1<t$, $t_2<t$ and $t_1+t_2 +2\geq t$ we have

$$\sum_{D^1_V(t_1,t_2)} \prob_p(\inspn_{U\to V_1}\circ \inspn_{U\to V_2}) = \bigo{n^{t-t_2}n^{i_2^2+2t_2-i_2}n^{f(t_1,r_1,x_1)}p^{i_2+1+x_1}}.$$

The exponent of $n^{t-t_2+i_2^2+2t_2-i_2+f(t_1,r,x_1)-f(t,r,x)}p^{i_2+x_1+1-x}$ in terms of $i_2$, $x_1$, $x$ and $r$ is bounded above by

\beqlbl \label{exub2}
i_2^2 + i_2 + \alpha_2 + (x_1-1)(x_1+r) + r+ 2x_1 -\beta_1 - (x-1)(x+r) -2j(i_2+x_1-x+1).
\eeqlbl

An increase in $i_1$ or $x_1$ will cause a decrease in this upper bound.  Therefore, we only need to show the above upper bound on the exponent is nonpositive for the smallest choices of $i_2$ and $x_1$.  This occurs when $i_2 + x_1+1 = x$ or $i_2+x_1=x$ depending on the parity of $t_2, t_1-r$ and $t-r$.  There are eight possible choices for the parity and in each case \eqref{exub2} is nonpositive.  We check the simplest case when all are even and $i_2 + x_1 +1= x.$  The \eqref{exub2} simplifies to 

$$i_2(i_2+1) + x_1(x_1+1) - (i_2+x_1)(i_2+x_1+1+r) \leq 0.$$

For the finite number of choices of $t_1$ and $t_2$ 

$$\sum_{D^1_V(t_1,t_2)} \prob_p(\inspn_{U\to V_1}\circ \inspn_{U\to V_2}) = \bigo{Q(t,r)}$$ so 
\beqlbl \label{d1vbound}
\sum_{D^1_V}\prob_p(\inspn_{U\to V_1}\circ \inspn_{U\to V_2})  = \bigo{Q(t,r)}.
\eeqlbl

Combining \eqref{d1vbound} and \eqref{d0vbound} finishes the proof.
\end{proof}  

\begin{proof}[Proof of Lemma \ref{intersection} (by induction on $t$ and $s$)]

Before we begin the proof we note that if $r = s = t$ then 
\beqlbl
p_{VW} = p_V = \bigo{n^{i^2+ 2t-i}p^{i+1}}
\eeqlbl
and if $r = s < t$ then
\beqlbl \label{contained}
p_{VW} \leq p_W \prob_p(\inspn_{W\to V}) = \bigo{ n^{f(t,r,x) + k^2+2s-k}p^{x+k+ 1}}
\eeqlbl
by applying Lemma \ref{condprob}.  

Therefore we only need to consider the case when $r< s\leq t$.  The symmetry of $V$ and $W$ will account for when $t< s$.  

Let $R = R(t,s,i,k,r,l) = n^{l^2 + l+r}n^{f(t,2l,i-l)}n^{f(s,2l,k-l)}p^{i+k-l+1}$ as in the statement of the lemma.  If $\inspn_{V} \cap \inspn_W$ occurs then for some pair $(V_1,V_2)\in D_V,$ $\left(\inspn_{V_1}\circ\inspn_{V_2}\right )\cap \inspn_W$ must occur.  We prove the lemma by showing

$$\frac{1}{R}\sum_{D_V} \prob_p\left ( (\inspn_{V_1}\circ \inspn_{V_2}) \cap \inspn_W \right) = \bigo{1}.$$

We use the same definitions of $t_1,t_2,i_1,i_2,$ etc. as in Lemma \ref{condprob} and also define:
\begin{itemize}
\item  $r_1 = \dim(V_1\cap W)$, $l_1 = \floor{r_1/2}$,
\item $r_2  = \dim(V_2\cap W)$, $l_2 = \floor{r_2/2}.$
\end{itemize}

If both $V_1$ and $V_2$ do not interect $W$ we has $\inspn_{V_1}\circ \inspn_{V_2}$ is independent of $\inspn_W$.  Let $D_{VW}^0$ denote sucha subset of $D_V$.  Then 

\beqlbl \label{sum0}
\sum_{D_{VW}^0}\prob_p\left( ( \inspn_{V_1}\circ\inspn_{V_2}) \cap \inspn_W\right) \leq \bigo{p_Vp_W}.
\eeqlbl

Similar to Lemma \ref{condprob} one can see that $p_Vp_W = \bigo{R(t,s,i,l,r,l)}.$  We are left with the two cases:  when only $V_1$ (w.l.o.g.) intersects $W$, and when both $V_1$ and $V_2$ intersect $W$.  

Recall $D_V(t_1,t_2)$ is the subset of $D_V$ such that $\dim(V_1) = t_1$ and $\dim(V_2) = t_2$.  Let $D_{VW}^2(t_1,t_2,r_1,r_2)$ denote the subset of $D_{V}(t_1,t_2)$ such that both $V_1$ and $V_2$ intersect $W$ and the dimension of the intersection is $r_1$ and $r_2$ respectively and $D_{VW}^1(t_1,t_2,r_1)$ denote the subset of $D_V(t_1,t_2)$ where the intersection of $V_1$ and $W$ has dimension $r_1$ and where $V_2$ does not intersect $W$.  

Let us first consider the sum 

\beqlbl\label{sum2}
\frac{1}{R}\sum_{D_{VW}^2(t_1,t_2,r_1,r_2)} \prob_p((\inspn_{V_1} \circ \inspn_{V_2}) \cap \inspn_W).
\eeqlbl

The summand satisfies both of the following inequalities:

$$ \prob_p((\inspn_{V_1} \circ \inspn_{V_2} )\cap \inspn_W) \leq \prob_p(\inspn_{V_1}\cap \inspn_{W})\prob_p(\inspn_{W\to V_2})$$ and
$$ \prob_p((\inspn_{V_1} \circ \inspn_{V_2} )\cap \inspn_W) \leq \prob_p(\inspn_{V_2}\cap \inspn_{W})\prob_p(\inspn_{W\to V_1}).$$

Let us assume (w.l.o.g.) that $V_2$ is not contained in $W$ and therefore $l_2< i_2$.  We may use the induction hypothesis and Lemma \ref{condprob} to show that for $(V_1,V_2) \in D^1_{VW}(t_1,t_2,r_1,r_2)$ 

$$\prob_p((\inspn_{V_1}\circ\inspn_{V_2})\cap \inspn_W) = \bigo{n^{l_1^2+ l_1 + r_1+ f(t_1,2l_1,i_1-l_1) +f(s,2l_1,k-l_1) + f(t_2,r_2,x_2)}p^{i_1+k-l_1+1 + x_2}}.$$

For each choice of $t_1,t_2,r_1,r_2$, the size of $D_{VW}^2(t_1,t_2,r_1,r_2)$ is $\bigo{n^{2r-r_1-r_2}}.$  Therefore to show
\beqlbl
\frac{1}{R}\sum_{D_{VW}^2(t_1,t_2,r_1,r_2)} \prob_p((\inspn_{V_1}\circ\inspn_{V_2})\cap \inspn_W) =\bigo{1}
\eeqlbl
it is sufficient to show
\begin{multline}\label{dvw2exp}
2r-r_1-r_2+l_1^2+l_1+r_1+f(t_1,2l_1,i_1-l_1) + f(s,2l_1,k-l_1)+f(t_2,r_2,x_2)\\ -l^2-l-r- f(t,2l,i-l)-f(t,2l,k-l) \\- 2j(i_1+k+x_2+1-l_1-i-k+l-1)
\end{multline}
is nonpositive when $D_{VW}^2(t_1,t_2,r_1,r_2)$ is nonempty.  The expression in \ref{dvw2exp} decreases with an increase $i_1$ or $i_2$, and also decreases with a coupled increase in both $l_1$ and $i_1$ or $l_2$ and $i_2$.  It suffices to consider minimal cases when $r_1 + r_2 +2 = r$ and $t_1 + t_2 = t-2$.   If the parity of all the variables is even then $i_1 + i_2 = i-1$, $l_1 + l_2 = l-1$ and $x_2 = i_2-l_2$ and \ref{dvw2exp} simplifies from

\begin{multline*}
2r-r_1-r_2+l_1^2+l_1 +r_1 + i_1^2-l_1^2-i_1-l_1 + t_1 + k^2-l_1^2-k-l_1 + s + i_2^2-l_2^2 - i_2-l_2\\
- l^2-l-r-i^2+l^2+i+l -t -k^2+l^2+k+l - s
\end{multline*}
to 
$$
r_1 +2 -2i - 2i_1i_2   +2l+ 2l_1l_2 \leq (2l_1( l_2 + 1)  - 2i_1i_2) + (2- 2(i-l))\leq 0 
$$
since $l<i$, $l_2 < i_2$ and $l_1 \leq i_1$.  

A similar computation shows that \ref{dvw2exp} is nonpositive for the other $2^6-1$ parity combinations.

Lastly we assume only $V_1$ interects $W$ and $V_2$ does not.  For $t_1 + t_2 + 2\geq t$ we have $$|D_{VW}^1(t_1,t_2,r_1)| = \bigo{n^{r-r_1 + t- t_2}}.$$  For each $t_1,t_2,$ and$r_1$


\begin{multline}\label{lastcase}
\frac{1}{R}\sum_{D_{VW}(t_1,t_2,r_1)}\prob_p\left( ( \inspn_{V_1}\circ \inspn_{V_2}) \cap W \right) \leq \prob_p( \inspn_{V_2}) \prob_p( \inspn_{V_1}\cap \inspn_W) \\
= \bigo{\frac{1}{R}n^{t-t_2+r-r_1}n^{i_2^2 + 2t_2 - i_2}p^{i_2+1} n^{ l_1^2 + l_1 + r_1 + f( t_1,2l_1,i_1-l_1)+ f( s,2l_1,k-l_1)}p^{k+i_1-l_1+1} }.
\end{multline}

With some simplifications the exponent in \ref{lastcase} is at most

\begin{multline}\label{dvw1exp}
t-t_2+r-r_1+i_2+2t_2-i_2 + \\l_1^2+r_1+l_1+i_1^2-l_1^2-i_1-l_1+t_1+ k^2-l_1^2-k-l_1+s\\
-l^2-l-r-i^2+l^2+l+i-t-k^2+l^2+k+l-s \\-2j(i_2+i_1+k-l_1+2-i-k+l-1).
\end{multline}

This decreases with increases in either $i_1$ or $i_2$ and also decreases with a coupled increase in $i_1$ and $l_1$.  Again there are parity choices for $t$, $r$, $t_1$, $t_2$,  and $r_1.$  Assume that each of the variables are even and minimal ($t_1+t_2 +2=t$ and therefore $i_1+i_2=i-1$).  Then \eqref{dvw1exp} simplifies to

\beqlbl\label{dvw1exp2}
-2i_1i_2+(l-l_1)(l+l_1+1-2j)< 0
\eeqlbl

A similar computation shows that \ref{dvw1exp} is nonpositive for the other $2^5-1$ possible choices for the parity of each of the variables.  

Altogether the three sums combine to show

$$ p_{VW}\leq \sum_{D_V}\prob_p((\inspn_{V_1}\circ \inspn_{V_2}) \cap \inspn_W) = \bigo{R}.$$
\end{proof}

\bibliographystyle{siam}

\end{document}